\documentclass[12pt,twoside,a4paper]{amsart}
\usepackage[centertags]{amsmath}
\usepackage{amssymb,amscd,amsfonts,latexsym,a4wide}
\usepackage{amsthm}
\usepackage{eucal}
\usepackage{array}
\usepackage{graphics}
\usepackage{amsmath}
\usepackage{amssymb}
\usepackage[latin2]{inputenc}

\renewcommand{\subjclass}[1]{\thanks{\emph{2010 Mathematics Subject Classification:}~#1}}
\renewcommand{\keywords}[1]{\thanks{\emph{Keywords and Phrases:}~#1}}
\renewcommand{\date}{\thanks{\today}}

\advance\headheight by 3pt

\newtheorem{thm}{Theorem}[section]
\newtheorem{lem}[thm]{Lemma}
\newtheorem{cor}[thm]{Corollary}

\newtheorem{thmp}[thm]{Proposition}

\theoremstyle{definition}

\newtheorem*{deff*}{Definition}

\newtheorem*{rem*}{Remark}

\newtheorem*{thmp*}{Proposition}

\newcommand{\thmref}[1]{Theorem \ref{#1}}

\newcommand{\propref}[1]{Proposition \ref{#1}}
\newcommand{\bnm}[2]{\mbox{$\textstyle{\binom{#1}{#2}}$}}

\newcommand{\fk}{{finite \'{e}tale $K$-algebra}}
\newcommand{\fek}{{finite \'{e}tale $K$-algebra }}
\newcommand{\Kbar}{\overline{K}}
\newcommand{\Qbar}{\overline{\mathbb{Q}}}
\newcommand{\kdots}{,\ldots ,}
\newcommand{\Deg}{{\rm Deg}\,}

\numberwithin{equation}{section}

\font\irott=rsfs10

\def\irF{\hbox{\irott F}}

\def\Qq{\mathbb{Q}}
\def\Zz{\mathbb{Z}}

\def\Ff{\mathbb{F}}

\def\O{O}
\def\M{\mathcal{M}}
\def\N{\mathcal{N}}

\def\house#1{\setbox1=\hbox{$\,#1\,$}%
\dimen1=\ht1 \advance\dimen1 by 2pt \dimen2=\dp1 \advance\dimen2
by 2pt
\setbox1=\hbox{\vrule height\dimen1 depth\dimen2\box1\vrule}%
\setbox1=\vbox{\hrule\box1}%
\advance\dimen1 by .4pt \ht1=\dimen1 \advance\dimen2 by .4pt
\dp1=\dimen2 \box1\relax}

\title[Discriminant equations]{Effective results for discriminant equations over finitely generated integral domains}

\author[J.-H. Evertse]{Jan-Hendrik Evertse}
\address{J.-H. Evertse\newline Leiden University, Mathematical Institute\newline
P.O.Box 9512, 2300 RA Leiden, The Netherlands}
\email{evertse@math.leidenuniv.nl}

\author[K. Gy\H{o}ry]{K\'{a}lm\'{a}n Gy\H{o}ry}
\address{K. Gy\H{o}ry\newline Institute of Mathematics, University of Debrecen\newline
H-4032 Debrecen, Egyetem T\'{e}r 1, Hungary}
\email{gyory@science.unideb.hu}

\begin{document}

\begin{abstract}
Let $A$ be an integral domain with quotient field $K$
of characteristic $0$ that is finitely generated
as a $\Zz$-algebra.
Denote by $D(F)$ the discriminant of a polynomial $F\in A[X]$.
Further, given a \fek $\Omega$, denote by $D_{\Omega /K}(\alpha )$
the discriminant of $\alpha$ over $K$.
For non-zero $\delta\in A$, 
we consider
equations
\[
D(F)=\delta
\]
to be solved in monic polynomials $F\in A[X]$ of given degree
$n\geq 2$ having their zeros in a given finite extension field $G$
of $K$, and
\[
D_{\Omega/K}(\alpha)=\delta\,\,\mbox{ in } \alpha\in O,
\]
where $O$ is an
$A$-order of $\Omega$, i.e., a subring of the integral closure of $A$ in
$\Omega$ that contains $A$ as well as a $K$-basis of $\Omega$.

In \cite[Chap. 10]{EvertseGyory16} we proved that if $A$ is effectively
given (in a well-defined sense described below) and integrally closed,
then up to natural notions of equivalence the above equations have only
finitely many solutions, and that moreover, a full system of representatives
for the equivalence classes can be determined effectively.
In the present paper, we extend these results to integral domains $A$
that are not necessarily integrally closed.
\end{abstract}

\date{}

\subjclass{Primary 11D99; Secondary 11D41}

\keywords{Discriminant equation, effective finiteness theorems}

\renewcommand{\baselinestretch}{1.1}
\maketitle

\section{Introduction}

Let $A$ be an integral domain of characteristic $0$ that is finitely generated
over $\Zz$ and denote by $K$ its quotient field.
We define the discriminant of a monic polynomial $F=X^n+a_1X^{n-1}+\cdots +a_n=(X-\alpha_1)\cdots (X-\alpha_n)$ by
\[
D(F):=\prod_{1\leq i<j\leq n} (\alpha_i-\alpha_j)^2.
\]
Recall that $D(F)\in\Zz [a_1\kdots a_n]$;
in fact, it is a polynomial of total degree $2n-2$ in $a_0\kdots a_n$.
Let $n$ be an integer $\geq 2$, let $\delta$ be a non-zero element of $A$ and 
let $G$ a finite extension of $K$.
We consider the equation
\begin{equation}\label{e10.1m}
\begin{array}{ll}
D(F)=\delta\ \ &\mbox{in monic polynomials $F\in A[X]$ of degree $n$}
\\
&\mbox{having all their zeros in $G$.}
\end{array}
\end{equation}
Two monic polynomials $F_1,F_2\in A[X]$ are called \emph{$A$-equivalent}
if $F_2(X)=F_1(X+a)$ for some $a\in A$. Clearly, $A$-equivalent polynomials
have the same discriminant, and so the solutions of equation
\eqref{e10.1m} can be divided into $A$-equivalence classes. 
We proved \cite[Thm. 10.1.1]{EvertseGyory16} the following result:
\\[0.2cm]
{\bf Theorem A.} \emph{Assume in addition to the above that $A$ is integrally closed.
Then the solutions of equation \eqref{e10.1m} lie in finitely many $A$-equivalence
classes. If moreover
$A$, $\delta$
and $G$ are effectively given (in a sense defined in the next section), 
then a full system of representatives of these $A$-equivalence classes
can be determined effectively.} 
\\[0.2cm]
In fact, this was a culmination of earlier results, due to Gy\H{o}ry,
who proved Theorem A first in the case $A=\Zz$
\cite{Gyory73}, then in the case that $A$ is the
ring of $S$-integers of a number field \cite{Gyory78a}, \cite{Gyory78b}, and 
then later for a special class of integral
domains containing transcendental elements \cite{Gyory84}.

In the present paper we prove a generalization of Theorem A, where instead
of integrally closed domains we consider 
integral domains $A$ such that
\begin{equation}\label{e10.1.2m}
\mbox{$\big(\frac{1}{n}A^+\cap A_K^+\big)/A^+$ is finite,}
\end{equation}
see Theorem \ref{t10.1.1m} below.
Here $B^+$ denotes the additive group
of a ring $B$, and $A_K$ denotes the integral closure of $A$ in $K$.
The class of domains with \eqref{e10.1.2m} contains all integrally 
closed domains,
the integral domains that contain $n^{-1}$, and also all finitely generated subrings
of $\Qbar$. We do not know if condition \eqref{e10.1.2m} 
is the weakest possible.
We will give an example of an integral domain $A$ and a field extension $G$
for which the finiteness part
of Theorem A is false. So some condition will have to be imposed on 
the domain $A$.

We also consider discriminant equations where the unknowns are elements
of orders of \fk s. Let for the moment $K$ be any field of characteristic $0$
and $\Omega$ a \emph{\fk ,} i.e., $\Omega = K[X]/(P)=K[\theta ]$,
where $P\in K[X]$ is some separable polynomial and $\theta := X\, {\rm mod}\, P$.
We write $[\Omega :K]:=\dim_K \Omega$; then clearly $[\Omega :K]=\deg P$.
Let $\Kbar$ be an algebraic closure of $K$.
By a \emph{$K$-homomorphism} of $\Omega$ to $\Kbar$ we mean
an non-trivial $K$-algebra homomorphism. 
There are precisely $n:=[\Omega :K]$ $K$-homomorphisms of $\Omega$ to $\Kbar$, which map $\theta$ to the $n$ distinct zeros of $P$ in $\Kbar$. 
We denote these by $x\mapsto x^{(i)}$ ($i=1\kdots n$). The \emph{discriminant}
of $\alpha\in\Omega$ over $K$ is given by
\[
D_{\Omega /K}(\alpha ):=\prod_{1\leq i<j\leq n}(\alpha^{(i)}-\alpha^{(j)})^2,
\]
where $\alpha^{(i)}$ denotes the image of $\alpha$ under $x\mapsto x^{(i)}$.
This is an element of $K$.
It is not difficult to show that $D_{\Omega /K}(\alpha +a)=D_{\Omega /K}(\alpha )$
for $\alpha\in\Omega$, $a\in K$. 
Further, $D_{\Omega /K}(\alpha )\not= 0$ if and only if $\Omega =K[\alpha ]$.
For more details on \fk s, we refer
to \cite[Chap. 1]{EvertseGyory16}.

Now let as above $A$ be an integral domain with quotient field $K$
of characteristic $0$ that is finitely generated over $\Zz$.
We denote by $A_{\Omega}$ the integral closure of $A$ in $\Omega$.
An \emph{$A$-order} is a subring of $A_{\Omega}$ that contains $A$, as well
as a $K$-basis of $\Omega$. 
In particular $A_{\Omega}$ itself is an $A$-order of $\Omega$.
We consider equations of the type
\begin{equation}\label{e10.1.6m}
D_{\Omega/K}(\alpha)=\delta\,\,\mbox{ in } \alpha\in \O
\end{equation}
where $\delta$ is a non-zero element of $A$, and $\O$ is an $A$-order of
$\Omega$. We call $\alpha_1 ,\alpha_2\in\O$ \emph{$A$-equivalent}
if $\alpha_1 -\alpha_2\in A$. 
Then the solutions of \eqref{e10.1.6m} can be divided into $A$-equivalence classes.
We recall Theorem 10.1.3 of \cite{EvertseGyory16}:
\\[0.2cm]
{\bf Theorem B.} \emph{Assume in addition to the above that $A$ is integrally closed.
Then the solutions
of equation \eqref{e10.1.6m} lie in finitely many $A$-equivalence classes.
If moreover $A$, $\delta$ and $\Omega$
are effectively given as defined in the next section
then a full system of representatives for these $A$-equivalence classes
can be determined effectively.}
\\[0.2cm]
Again, Gy\H{o}ry \cite{Gyory74}, \cite{Gyory76}, \cite{Gyory78a}, \cite{Gyory78b}, \cite{Gyory84} proved this result earlier in special cases.
In this paper we extend Theorem B to integral domains $A$
such that
\begin{equation}\label{e10.1.5m}
(\O\cap K)^+/A^+\ \ \mbox{is finite,}
\end{equation}
see Theorem \ref{t10.1.4m} below.
Notice that $O\cap K =A$ if
$A$ is integrally closed.
It is shown that the finiteness result becomes false if we weaken condition
\eqref{e10.1.5m}.

In Section \ref{section10.1m} we give the precise statements of our results.
Section \ref{section10.3m} contains some tools from effective commutative
algebra needed in this paper. Much of this has been taken from 
\cite[\S10.7]{EvertseGyory16}.
In
Section \ref{section10.2m} we recall from \cite[Chap. 10]{EvertseGyory16}
a central proposition,
which is the basis of the proofs of the results in the present paper.
The main tool in the proof of that Proposition is Corollary 1.2 of \cite{EvertseGyory13}
on unit equations over finitely generated integral domains.
In the remaining sections we deduce our theorems. 

\section{Statements of the results}\label{section10.1m}

We start with the necessary definitions.
Let $A$ be an integral domain of characteristic $0$ which is finitely generated over $\Zz$ and $K$
its quotient field. Suppose $A=\Zz[x_1,\ldots ,x_r]$ and define the ideal
\[
I:=\{ f\in\Zz [X_1\kdots X_r]:\, f(x_1\kdots x_r)=0\}.
\]
Thus, $A$ is isomorphic to $\Zz[X_1,\ldots
,X_r]/I$ and $x_i$ corresponds to the residue class of $X_i$ mod
$I$. Following \cite[\S10.7]{EvertseGyory16}, we say that $A$ is
\textit{given effectively} if a finite set of generators for the
ideal $I$ is given. We call such a set of generators for $I$
an \emph{ideal representation} for $A$.
By a representation for 
an element $y$ of $A$ we mean an polynomial $f\in\Zz [X_1\kdots X_r]$
such that $y=f(x_1\kdots x_r)$ and we say that $y$ is 
\emph{effectively given/computable} if such $f$ is given/can be computed.
We can check whether two polynomials $f,g\in\Zz [X_1\kdots X_r]$
represent the same element of $A$ by checking whether $f-g\in I$, 
using an ideal membership algorithm for $I$ (see e.g., \cite{Aschenbrenner04}). 
By a representation for $y\in K$ we mean a pair $(f,g)$
with $f,g\in\Zz [X_1\kdots X_r]$ and $g\not\in I$ such that
$y=f(x_1\kdots x_r)/g(x_1\kdots x_r)$.
By saying that a polynomial with coefficients in $A$
is given/can be determined effectively we mean that its coefficients are given/can be determined effectively.

In what follows, $A$ is an integral domain finitely generated over $\Zz$,
$K$ its quotient field and $G$ a finite
extension of $K$. Further,
$\delta$ is a non-zero element of $A$ and $n$ an
integer with $n\geq 2$. Consider the equation
\begin{equation*}\tag{\ref{e10.1m}}
\begin{array}{ll}
D(F)=\delta\ \ &\mbox{in monic polynomials $F\in A[X]$ of degree $n$}
\\
&\mbox{having all its zeros in $G$.}
\end{array}
\end{equation*}
Our first result is as follows.

\begin{thm}\label{t10.1.1m}
Let $n$ be an integer $\geq 2$ and
$A$ an integral domain finitely generated over $\Zz$ with quotient field $K$ such that
\begin{equation*}\tag{\ref{e10.1.2m}}
\mbox{$\big(\frac{1}{n}A^+\cap A_K^+\big)/A^+$ is finite.}
\end{equation*}
Further, let $G$ be a finite extension of $K$
and $\delta$ a non-zero element of $A$.
\\[0.1cm]
Then the
set of monic polynomials $F\in A[X]$ with \eqref{e10.1m} is a
union of finitely many $A$-equivalence classes. Moreover,
for any effectively given $n$, $A$, $G$, $\delta$ as above, a
full system of representatives for these equivalence classes can
be determined effectively.
\end{thm}

By Corollary \ref{cor:E0-18} in Section \ref{section10.3m} below,
for any effectively given integral domain $A$ finitely generated over $\Zz$
it can be decided effectively whether it satisfies \eqref{e10.1.2m}.
As said in the introduction,
we do not know whether condition \eqref{e10.1.2m} can be
relaxed. 
Below, we show that Theorem \ref{t10.1.1m} is not true
for arbitrary finitely generated domains of characteristic $0$.

We now turn to elements of orders in finite \'{e}tale algebras.
We start again with some definitions.
Let again $A$ be an integral domain finitely generated over $\Zz$
and $K$ its quotient
field.
A finite \'{e}tale $K$-algebra
$\Omega$ (or in particular a finite field extension of $K$) is given effectively,
if a monic polynomial
$P\in K[X]$ without multiple zeros is given effectively such that 
$\Omega= K[X]/(P)$.
Using \cite[Thm. 10.7.5]{EvertseGyory16} it can be decided effectively whether $P$ is irreducible, and thus,
whether $\Omega$ is a field.
Elements of $\Omega$ can be expressed in the form
$\sum_{i=0}^{n-1}a_i\theta^{i}$ with $a_0,\ldots ,a_{n-1}\in
K$, where $\theta :=X\ {\rm mod}\, P$.
We say that an element of $\Omega$ is given/can be determined effectively
effectively if $a_0,\ldots, a_{n-1}$ are given/can be determined
effectively.

Recall that an $A$-order of $\Omega$ is an
$A$-subalgebra of the integral closure of $A$ in $\Omega$, which
spans $\Omega$ as a $K$-vector space. By a result of Nagata
\cite{Nagata56}, the integral closure of $A$ in
$\Omega$ is finitely generated as an $A$-module. Since the integral domain
$A$ is Noetherian, any $A$-order of $\Omega$ is finitely generated
as an $A$-module as well. We say that an $A$-order $O$ of $\Omega$ is given effectively
if a finite set of $A$-module generators 
$\left\{\omega_1=1,\omega_2,\ldots ,\omega_m\right\}$ of $O$ 
is given effectively.
We say that an element $\alpha$ of $O$ is
given/can be determined effectively, if
$a_1,\ldots ,a_m\in A$ are given/can be determined effectively
such that $\alpha=\sum_{i=1}^{m}a_i\omega_i$.
Using Corollary \ref{cor:E0-19} (i) below one can verify that
$\omega_1,\ldots ,\omega_m$ do indeed generate an $A$-order of
$\Omega$.

Let $\Omega$ be a finite \'{e}tale $K$-algebra with
$[\Omega :K]=:n\geq 2$, let $O$ be an $A$-order in $\Omega$,
and let $\delta$ be a non-zero element of $A$.
We consider the equation
\begin{equation*}\tag{\ref{e10.1.6m}}
D_{\Omega/K}(\alpha)=\delta \,\mbox{ in } \alpha\in O.
\end{equation*}
We prove the following result.

\begin{thm}\label{t10.1.4m}
Let $A$ be an integral domain finitely generated over $\Zz$ with quotient field $K$,
$\Omega$ a finite \'{e}tale $K$-algebra, $O$ an $A$-order in $\Omega$,
and $\delta\in A$, $\delta\not= 0$. Assume that
\begin{equation*}\tag{\ref{e10.1.5m}}
\mbox{$(O\cap K)^+/A^+$ is finite.}
\end{equation*}
Then the set of $\alpha\in O$ with \eqref{e10.1.6m} is a union of
finitely many $A$-equivalence classes.
Further, for any effectively given $A$, $\Omega$, $O$, $\delta$ as above,
a full system of representatives for these classes can be determined
effectively.
\end{thm}

Using Corollary \ref{cor:E0-19} (ii) below, for given $A$, $\Omega$, $O$ it
can be decided effectively whether condition \eqref{e10.1.5m} is
satisfied.
At the end of the present section we show that Theorem \ref{t10.1.4m}
becomes false if we relax condition \eqref{e10.1.5m}.

We now show that Theorem \ref{t10.1.1m} cannot be true for arbitrary
finitely generated domains $A$.
More precisely, we show that for
every integer $n\geq 2$, there are an integral domain $A$ finitely generated over $\Zz$,
a finite extension $G$ of the quotient field of $A$, and $\delta\in A\setminus\{ 0\}$,
such that there are infinitely many $A$-equivalence classes
of monic polynomials $F\in A[X]$ of degree $n$ with $D(F)=\delta$ having all their roots in $G$.

Let $n$ be an integer $\geq 2$,
let $t$ be transcendental over $\Qq$ and define the integral domain
\[
A:=\Zz \left[nt,\bnm{n}{2}t^2,\bnm{n}{3}t^3,\ldots, t^n\right].
\]
Notice that $A$ is a subring of $\Zz [t]$ and that
$A$ has quotient field $K:=\Qq (t)$.
We can express elements of $A$ as $\sum_{k=0}^m s_kt^k\in A$ with $s_k\in\Zz$
for all $k$. We show that
$s_k$ is divisible by $n$ if $k$ is coprime with $n$.
Indeed, $s_k$ is a $\Zz$-linear combination
of terms
\begin{equation}\label{e10.5.100m}
\prod_{j=1}^n\bnm{n}{j}^{l_j}\ \ \mbox{with $l_1\kdots l_n\in\Zz_{\geq 0}$,\ \ $l_1+2l_2+\cdots +nl_n=k$.}
\end{equation}
Let $p^r$ be a prime power occurring in the prime factorization
of $n$.
For each term in \eqref{e10.5.100m}, there is $j\in\{ 1,\ldots, n-1\}$
such that $j$ is coprime with $p$ and $l_j>0$,
since $k$ is not divisible by $p$.
From well-known divisibility properties of binomial coefficients,
it follows that $\bnm{n}{j}$ is divisible by $p^r$.
Hence all terms in \eqref{e10.5.100m} are divisible by $p^r$.
Consequently, $s_k$ is divisible by
each of the prime powers $p^r$ in the factorization of $n$,
hence it is divisible by $n$.

Now fix a non-zero $c\in A$, let $\delta$ be the discriminant of $X^n-c$,
and let $G$ be the splitting field of $X^n-c$ over $K$. Consider the
polynomials
\[
F_m:= (X+t^{mn+1})^n-c = \sum_{j=0}^n \bnm{n}{j}t^j\cdot (t^n)^{mj}X^{n-j}\, -c
\]
where $m$ runs through the positive integers.
Clearly, for every $m$ we have $F_m\in A[X]$,
$F_m$ has splitting field $G$ over $K$,
and $D(F_m)=\delta$. We show that the polynomials $F_m$ lie in distinct
 $A$-equivalence classes; it then follows that \eqref{e10.1m} has
infinitely many $A$-equivalence classes of solutions.
Let $m,m'$ be two distinct positive integers. Suppose that $F_m,F_{m'}$ are $A$-equivalent; then there is $a\in A$ such that
$F_{m'}(X)=F_m(X+a)$. It follows that there is an $n$-th root of unity $\rho$
such that $X+t^{m'n+1}=\rho (X+t^{mn+1}+a)$. Consequently,
$\rho =1$ and $t^{m'n+1}-t^{mn+1}\in A$.
But this is impossible, since the exponents of both terms are coprime
with $n$, while the coefficients are not divisible by $n$.

We finish this section by showing that if, with the notation of  
Theorem \ref{t10.1.4m},
the integral domain $A$ and the $A$-order $O$
of $\Omega$ do not satisfy \eqref{e10.1.5m},
then there is a non-zero
$\delta\in A$ such that \eqref{e10.1.6m} has infinitely many $A$-equivalence classes of solutions. Indeed,
suppose \eqref{e10.1.5m} does not hold. Then there is an infinite
sequence $b_1,b_2,\ldots$ of elements of $O\cap K$ such that none
of the differences $b_i-b_j$ ($i>j\geq 1$) belongs to $A$.
Pick
$\alpha\in O$ such that $\Omega =K [\alpha ]$ and put
$\delta :=D_{\Omega /K}(\alpha )$. Then $D_{\Omega/K}(\alpha +b_i)=\delta$
and $\alpha +b_i\in O$
for $i=1,2,\ldots$, and the elements $\alpha +b_i$ ($i=1,2,\ldots$)
lie in different $A$-equivalence classes.

\section{Tools from effective commutative algebra}\label{section10.3m}

For the definitions of a domain, \'{e}tale algebra, order, etc. 
and elements of those being effectively given/computable we refer
to Section \ref{section10.1m}. 
We start with some effective results on systems of linear equations 
in polynomials.

\begin{thmp}\label{th:E0-1}
Let ${\bf k}=\Qq$ or $\Ff_p$ for some prime $p$.
Then for any given positive integer $r$ and any given polynomials
$f_1\kdots f_s\in {\bf k}[X_1\kdots X_r]$ we can:
\\
(i) determine effectively whether a given polynomial $g$ from ${\bf k}
[X_1\kdots X_r]$ belongs to the ideal $I=(f_1\kdots f_s)$ and if
so, determine effectively polynomials $g_1\kdots g_s$ such that
$g=g_1f_1+\cdots +g_sf_s$ (ideal membership problem);
\\
(ii) determine effectively whether $I$ is a prime ideal.
\end{thmp}

\begin{proof}
See Seidenberg \cite{Seidenberg74}:
\S4, p. 277 for (i) and \S46, p. 293 for (ii) 
(in fact Seidenberg gives a method
to determine the prime ideals associated to a given ideal $I$,
which certainly enables one to decide whether $I$ is a prime
ideal).
The main ideas in the proofs of these results originate from
Hermann \cite{Hermann26} but her arguments contain gaps. 
\end{proof}

For a polynomial $f$ with integer coefficients, we denote by
$H(f)$ its height (maximum of the absolute values of its
coefficients) and by $\Deg f$ its total degree. Further, we define
the polynomial ring $R:=\Zz [X_1\kdots X_r]$.

\begin{thmp}\label{th:E0-9}
Let $M$ be an $m\times n$-matrix with entries from $R$, and ${\bf
b}$ a vector from $R^m$, such that the entries of $M$ and ${\bf
b}$ have total degrees at most $d$ and heights at most $H$.
\\[0.1cm]
(i) The $R$-module
\[
\{ {\bf x}\in R^n:\, M{\bf x}={\bf 0}\}
\]
is generated by vectors, of which the coordinates are polynomials,
whose total degrees  are bounded above by an effectively
computable number $C_1$ depending only on $m,n,d,r$ and whose
heights are bounded above by an effectively computable number
$C_2$ depending only on $m,n,d,r$ and $H$.
\\[0.1cm]
(ii) Suppose that the system
\[
M{\bf x}={\bf b}
\]
is solvable in ${\bf x}\in R^n$. Then this system has a solution
${\bf x}_0\in R^n$ whose coordinates have total degrees bounded
above by $C_3$ and heights bounded above by $C_4$, where both
$C_3,C_4$ are effectively computable numbers depending only on
$m,n,d,r$ and $H$.
\end{thmp}

\begin{proof}
Aschenbrenner \cite{Aschenbrenner04} proved the above with
$C_1=(2md)^{2^{c_1r\log 2r}}$,
$C_2=\exp\Big((2m(d+1))^{2^{c_2(1+r\log 2r)}}(1+\log H)\Big)$ (cf. his
Proposition 5.2) and $C_3=(2md)^{2^{c_3r\log 2r}}(1+\log H)$ (cf. his
Theorem 6.1), where $c_1,c_2,c_3$ are effectively computable
absolute constants. In (ii), thanks to our upper bound for the
total degrees,
the problem to find a solution to $M{\bf x}={\bf b}$
reduces to  solving a finite system of inhomogeneous linear
equations over $\Zz$, and then we obtain a value for $C_4$ for
instance by invoking for instance a result from 
\cite{Borosh89}.
\end{proof}

\begin{cor}[Ideal membership over $\Zz$]\label{cor:E0-10}
Let $I=(f_1\kdots f_s)$ be an ideal in $R$ and $g\in R$. Suppose
that $f_1\kdots f_s$ and $g$ have total degrees at most $d$ and
heights at most $H$. If $g\in I$, there exist $g_1\kdots g_s\in R$
of total degrees and heights bounded above by effectively
computable numbers depending only on $r$, $d$ and $H$, such that
$g=\sum_{i=1}^s g_if_i$.
\end{cor}

\begin{proof}
Apply part (ii) of Theorem \ref{th:E0-9} with $m=1$.
\end{proof}

In what follows, $A$ is an integral domain with quotient field $K$ of characteristic $0$ that is finitely generated over $\Zz$. We assume that
$A$ is effectively given, i.e., we have $A=\Zz [x_1\kdots x_r]$,
and we are given a finite set of generators $f_1\kdots f_s$ of the ideal
\[
I=\{ f\in\Zz [X_1\kdots X_r]:\, f(z_1\kdots f_r)=0\}.
\]

\begin{cor}\label{cor:E0-11}
Given an $m\times n$-matrix $M$ with entries in $K$ and ${\bf b}\in K^m$ one 
can:
\\[0.15cm]
(i) effectively determine a finite set of $A$-module generators
${\bf a}_1\kdots {\bf a}_t$ for the $A$-module of ${\bf x}\in A^n$
with $M{\bf x}={\bf 0}$;
\\ [0.15cm]
(ii) decide effectively whether $M{\bf x}={\bf b}$ has a solution 
${\bf x}\in A^n$ and if so, find a solution.
\end{cor}

\begin{proof}
(i) After clearing denominators, one may assume that the entries of $M$
and the coordinates of ${\bf b}$ lie in $A$.
Let $m_{ij}\in R$ ($i=1\kdots m$, $j=1\kdots n$ be representatives for the
elements of $M$. Writing $y_1\kdots y_n$ for representatives in $R$ for the 
coordinates of ${\bf x}$ we can rewrite the system $M{\bf x}={\bf 0}$ as
\[
m_{i1}y_1+\cdots +m_{in}y_n= f_1y_{i1}+\cdots +f_sy_{is}\ \ (i=1\kdots m)
\]
in $y_i,y_{ij}\in R$, which is a system of equations
as in part (i) of Proposition \ref{th:E0-9}.
Likewise $M{\bf x}={\bf b}$ can be rewritten as a system of equations
as in part (ii) of Proposition \ref{th:E0-9}. Now one simply has to apply
Proposition \ref{th:E0-9} to these systems.
\end{proof}

We say that a finitely generated $A$-module $\mathcal{M}\subset K$
is effectively given if a finite set of $A$-module generators for $\mathcal{M}$
is effectively given.
We denote the $A$-module generated by $a_1\kdots a_u$ by
$(a_1\kdots a_u)$.

\begin{cor}\label{cor:E0-12}
For any two effectively given $A$-submodules $\mathcal{M}_1$, $\mathcal{M}_2$
of $K$, one can
\\[0.15cm]
(i) effectively decide whether $\mathcal{M}_1\subseteq\mathcal{M}_2$;
\\[0.15cm]
(ii) effectively compute a finite set of $A$-module generators for
$\mathcal{M}_1\cap\mathcal{M}_2$.
\end{cor}

\begin{proof}
Let $\mathcal{M}_1=(a_1\kdots a_u)$, $\mathcal{M}_2=(b_1\kdots b_v)$
with the $a_i,b_j\in K$ effectively given.
Then (i) comes down to checking
whether $a_1\kdots a_u\in\mathcal{M}_2$, which is a special case
of part (ii) of Corollary \ref{cor:E0-11}. To determine a finite set of
$A$-module generators for $\mathcal{M}_1\cap\mathcal{M}_2$, one first
determines a finite set of $A$-module generators for the solution set
of $(x_1\kdots x_u,y_1\kdots y_v)\in A^{u+v}$ of
$\sum_{i=1}^u x_ia_i=\sum_{j=1}^v y_jb_j$ and then for each generator one
takes the coordinates $x_1\kdots x_u$. 
\end{proof}

Probably the following results is well-known but we could not find a
proof for it.

\begin{thmp}\label{th:E0-16}
Assume that $A$ is effectively given and let 
$\mathcal{M}_1,\, \mathcal{M}_2$ be two effectively given finitely generated
$A$-submodules of $K$ with $\mathcal{M}_1\subset\mathcal{M}_2$. 
Then it can be decided effectively whether
$\mathcal{M}_2/\mathcal{M}_1$ is finite.
If this is the case, a full system of representatives for $\mathcal{M}_2/\M_1$
can be determined effectively.
\end{thmp}

We use the following simple lemma.

\begin{lem}\label{lem:E0-17}
Suppose we are given a sequence
$\N_1\subseteq\cdots \subseteq\N_r$ of finitely generated $A$-modules
contained in $K$.
Then $\N_r/\N_1$ is finite if and only if for $i=1\kdots r-1$,
the quotient $\N_{i+1}/\N_i$ is finite.
Further, if this is the case, we obtain a full system of representatives
for $\N_r/\N_1$ by taking all sums $a_1+\cdots +a_{r-1}$
where $a_i$ runs through
a full system of representatives for $\N_{i+1}/\N_i$ for $i=1\kdots r-1$.
\end{lem}

\begin{proof}
Obvious.
\end{proof}

\begin{proof}[Proof of Proposition \ref{th:E0-16}]
We may assume that $A$ is given in the form
\[ 
\Zz [X_1\kdots X_r]/(f_1\kdots f_s),
\]
with given polynomials $f_1\kdots f_s\in\Zz [X_1\kdots X_r]$, and that $x_i$
is the residue class of $X_i$ modulo $(f_1\kdots f_s)$, for $i=1\kdots r$.
Then the elements of $K$ may be represented as quotients
$g(x_1\kdots x_r)/h(x_1\kdots x_r)$,
where $g,h\in\Zz [X_1\kdots X_r]$ and $h\not\in (f_1\kdots f_s)$.
After multiplying the given generators of $\M_1$ and $\M_2$ with the product
of their denominators, we may assume that
$\M_1,\, \M_2\subseteq A$.
There is clearly no loss of generality to assume that
$\M_1$, $\M_2$ are given as
$\M_1=(a_1\kdots a_u)$, $\M_2=(a_1\kdots a_v)$ with $v>u$.
In fact, it suffices to prove our Theorem in the special case $v=u+1$.
Then the general case with arbitrary $v$ can be deduced from
Lemma \ref{lem:E0-17}.

So we assume henceforth that $v=u+1$. Let
\[
J:=\{ x\in A:\, x\cdot a_{u+1}\in \M_1\}= A\cap a_{u+1}^{-1}\M_1;
\]
then $\M_2/\M_1$ is isomorphic to the additive group of $A/J$.
By Corollary \ref{cor:E0-12} we can compute a finite set of generators for $J$,
which we may represent as residue classes modulo $(f_1\kdots f_s)$ of
polynomials $f_{s+1}\kdots f_t$ from $\Zz [X_1\kdots X_r]$.
Then $\M_2/\M_1\cong \Zz [X_1\kdots X_r]/I$, where
$I=(f_1\kdots f_s\kdots f_t)$.
So it suffices to prove that it can be decided effectively whether
$\Zz [X_1\kdots X_r]/I$ is finite and that in this case a full system
of representatives can be computed effectively.

A necessary condition for $\Zz [X_1\kdots X_r]/I$ to be finite is that
$I\cap\Zz \not= (0)$. This in turn is equivalent
to the existence of
$g_1\kdots g_t\in\Qq [X_1\kdots X_r]$ such that $g_1f_1+\cdots +g_tf_t=1$.
By Proposition \ref{th:E0-1} it can be decided effectively whether
such $g_1\kdots g_t$ exist and if so, they can be computed.
Supposing such $g_1\kdots g_t$ exist,
by clearing the denominators of their coefficients
we find non-zero $b\in\Zz$ in $I\cap\Zz$. Using Corollary \ref{cor:E0-10}
we can check, for every divisor $a\in\Zz$ of $b$ whether $a\in I$.
In this manner we eventually find $a$ with $I\cap\Zz =(a)$.

If $a=1$ then $I=\Zz [X_1\kdots X_r]$
and we are done. Suppose that $a\not=1$.
We make a reduction to the case that $a=p$ is a prime number.
Suppose that $a=p_1\cdots p_k$
where $p_1\kdots p_k$ are prime numbers.
We may write $I=(p_1\cdots p_k,f_1\kdots f_t)$.
For $i=1\kdots k$, put
$I_i:=(p_1\cdots p_i,f_1\kdots f_t)$ and for
$i\in\{ 1\kdots k-1\}$ define
\[
J_i:= \{ f\in\Zz [X_1\kdots X_r]:\, p_1\cdots p_if\subseteq I_{i+1}\}.
\]
Then $I_i/I_{i+1}$ is isomorphic to the
additive group of $\Zz [X_1\kdots X_r]/J_i$.
Now if we are able to decide, for $i=1\kdots k-1$,
whether $\Zz [X_1\kdots X_r]/J_i$ is finite
and find a full system of representatives for this quotient,
we can do the same for $I_i/I_{i+1}$ and then, thanks to Lemma \ref{lem:E0-17},
for $\Zz [X_1\kdots X_r]/I$.

Using Theorem \ref{th:E0-9} we find a set of generators for $J_i$.
By what has been explained above, from this we can compute $b_i\in\Zz$
with $J_i\cap\Zz =(b_i)$.
Clearly, $p_{i+1}\in J_i$; hence $J_i\cap\Zz =(1)$ or $(p_{i+1})$.
The case $J_i=(1)$ being obvious,
it remains to check whether
$\Zz [X_1\kdots X_r]/J_i$ is finite if $J_i\cap\Zz =(p_{i+1})$.

Changing, notation, we see that it suffices to show, for any given ideal
$I$ of $\Zz [X_1\kdots X_r]$ with $I\cap\Zz =(p)$ for some prime $p$,
whether $\Zz [X_1\kdots X_r]/I$ is finite and if so, to compute a full system
of representatives for $\Zz [X_1\kdots X_r]$ modulo $I$.
We may assume that $I$ is given in the form $I=(p,f_1\kdots f_t)$,
with $f_1\kdots f_t\in\Zz [X_1\kdots X_r]$.
Given $f\in\Zz [X_1\kdots X_r]$, denote by $\overline{f}$ its reduction
modulo $p$, and put $\overline{I}=(\overline{f}_1\kdots\overline{f}_t)$.
Then $\Zz [X_1\kdots X_r]/I \cong \Ff_p[X_1\kdots X_r]/\overline{I}$.
So we have to decide whether this latter residue class ring is finite and if so,
to compute a full system of representatives for the residue classes.

For any positive integer $m$, denote by $V_m$ the set of residue classes
modulo $\overline{I}$ of all polynomials of degree $\leq m$ in
$\Ff_p[X_1\kdots X_r]$. This is a finite dimensional $\Ff_p$-vector space.
Recall that the \emph{Hilbert function} $H_{\overline{I}}$ of $\overline{I}$ is defined
by $H_{\overline{I}}(m):=\dim_{\Ff_p} V_m$.
It is known that there are an integer $m_{\overline{I}}$, and a polynomial $p_{\overline{I}}\in\Qq [X]$,
called the \emph{Hilbert polynomial} of $\overline{I}$,
such that $H_{\overline{I}}(m)=p_{\overline{I}}(m)$ for $m\geq m_{\overline{I}}$.
Now $\Ff_p[X_1\kdots X_r]/\overline{I}$ is finite if and only if $p_{\overline{I}}$ is constant, and this being the case, every residue class of
$\Ff_p[X_1\kdots X_r]$ modulo $\overline{I}$ is represented by a polynomial
of degree at most $m_{\overline{I}}$.
There is a general procedure, based on Gr\"{o}bner basis theory,
to compute $m_{\overline{I}}$ and $p_{\overline{I}}$, given a set of generators for
$\overline{I}$,
see \cite[\S\S15.1.1, 15.10.2]{Eisenbud94}. With this procedure one can decide
whether $\Ff_p[X_1\kdots X_r]/\overline{I}$ is finite. Subsequently, using Proposition \ref{th:E0-1}, one can select a full system of representatives modulo
$\overline{I}$ from the polynomials of degree $\leq m_{\overline{I}}$.

This completes the proof of Proposition \ref{th:E0-16}.
\end{proof}

For a finite extension $G$ of $K$, we denote by $A_G$ the integral closure
of $A$ in $G$. In particular, $A_K$ is the algebraic closure of $A$ in its
quotient field $K$. 

\begin{thmp}\label{cor:E0-17}
Assume that $A$ and a finite extension $G$ of $K$ are effectively given.
Then one can effectively compute a finite set of $A$-module generators
for $A_G$. Moreover, one can compute an ideal representation for $A_G$.
\end{thmp}

\begin{proof}
Computing a finite set of $A$-module generators for $A_G$
follows from results
of Nagata \cite{Nagata56}, de Jong \cite{deJong98}, Matsumura \cite{Matsumura86}
and Matsumoto \cite{Matsumoto00}.
For more details, see \cite[Cor. 10.7.18]{EvertseGyory16}.
Then an ideal representation for $A_G$ can be computed using
\cite[Thms. 10.7.13, 10.7.16]{EvertseGyory16}.
\end{proof}

\begin{cor}\label{cor:E0-18}
Assume that $A$ is effectively given. Then one can can effectively decide
whether $(\frac{1}{n}A^+\cap A_K^+)/A^+$ is finite and if so,
compute a full system of representatives for $(\frac{1}{n}A^+\cap A_K^+)/A^+$.
\end{cor}

\begin{proof}
Immediate consequence of Proposition \ref{cor:E0-17},
Corollary \ref{cor:E0-12}, (ii) and
Proposition \ref{th:E0-16}.
\end{proof}

\begin{cor}\label{cor:E0-19}
Assume that $A$ and a \fek $\Omega$ are effectively given.
Further, let $\omega_2\kdots\omega_u\in\Omega$ be effectively given
and let $O$ be the $A$-module generated by $1,\omega_2\kdots \omega_m$.
\\[0.15cm]
(i) It can be effectively decided whether $O$ is an $A$-order of $\Omega$.
\\[0.15cm]
(ii) If $O$ is an $A$-order of $\Omega$, one effectively decide
whether  $(O\cap K)^+/A^+$ is finite, and if so, compute a full system of representatives for $(O\cap K)^+/A^+$.
\end{cor}

\begin{proof}
We assume that $\Omega =K[X]/(P)$ with $P\in K[X]$ an effectively given,
separable monic polynomial. Let $n:=[\Omega :K]=\deg P$ and $\theta := X\, {\rm mod}\, P$. Then $\{1,\theta\kdots \theta^{n-1}\}$ is a $K$-basis of
$\Omega$. Further, we assume that $\omega_2\kdots\omega_u$ are effectively 
given as $K$-linear combinations of $1,\theta\kdots \theta^{n-1}$.
Then we may express elements of $O$ as
$\sum_{k=0}^{n-1} l_k({\bf x})\theta^k$ with ${\bf x}\in A^m$, 
where $l_0\kdots l_{n-1}$ are linear forms from $K[X_1\kdots X_m]$.

(i) We first verify that the linear forms $l_0\kdots l_{n-1}$ have rank $n$
over $K$, to make sure that $O$ contains a $K$-basis of $\Omega$.
The next thing to verify is    
whether $\omega_i\omega_j$ is an $A$-linear combination
of $1,\omega_2\kdots\omega_u$ for $i,j=2\kdots u$. Compute $b_{ij}\in K$
such that $\omega_i\omega_j=\sum_{k=0}^{n-1}b_{ijk}\theta^k$. Then we have to verify
whether the system $l_k({\bf x})=b_{ijk}$ ($k=0\kdots n-1$) 
is solvable in ${\bf x}\in A^u$, for $i,j=2\kdots u$, and this can be done
by means of 
Corollary \ref{cor:E0-11} (ii). Lastly, it is a standard fact from algebra,
that if $A$ is a subring of a commutative ring $B$ that is finitely generated 
as an $A$-module, then $B$ is in fact integral over $A$. So in particular,
if we have verified that $O$ is closed under multiplication then it is 
automatically contained in $A_{\Omega}$.

(ii) Using Corollary \ref{cor:E0-11} (i) we can compute a finite set of
$A$-module generators, say ${\bf x}_1\kdots {\bf x}_v$ for the $A$-module
of ${\bf x}\in A^u$ with
$l_i({\bf x})=0$ for $i=1\kdots n-1$. Then $(O\cap K)^+$ is generated as an
$A$-module by $l_0({\bf x}_1)\kdots l_0({\bf x}_v)$. With these
generators for $(O\cap K)^+$ and Proposition \ref{th:E0-16},
we can check whether $(O\cap K)^+/A^+$ is finite, and if so, compute a full system of representatives.
\end{proof}

\section{The main proposition}\label{section10.2m}

We recall from \cite{EvertseGyory16} a central proposition from which 
Theorems \ref{t10.1.1m} and \ref{t10.1.4m} are deduced.
We keep the
notation from Section \ref{section10.1m}.

\begin{thmp}\label{p10.2.2m}
For any integral domain $A$ of characteristic $0$ that is
finitely generated over $\Zz$,
any finite extension $G$ of the quotient field of $A$,
any non-zero $\delta\in A$, and any integer $n\geq 2$,
all effectively given, one can
determine effectively a finite subset $\irF =\irF_{A,G,n,\delta}$ of
$G$ with the following property: if $F$ is any monic polynomial
from $A[X]$ of degree $n$ and discriminant $\delta$ having all its
zeros, say $\alpha_1,\ldots ,\alpha_n$, in $G$, then
\begin{equation}\label{e10.2.1m}
\alpha_i-\alpha_j\in\irF\,\,\mbox{ for } i,
j\in\left\{1,\ldots, n\right\}, i\neq j.
\end{equation}
\end{thmp}

\begin{proof}
This is Proposition 10.2.1 of \cite{EvertseGyory16}.
Its proof is based on
Corollary 1.2 of \cite{EvertseGyory13} on
unit equations over finitely generated integral domains. 
\end{proof}

\section{Proof of Theorem \ref{t10.1.1m}}\label{section10.5m}

We start with a preliminary lemma.

\begin{lem}\label{l10.5.1m}
For every integral domain $A$ finitely generated over $\Zz$ and every
two monic polynomials $F_1$, $F_2\in A[X]$ with at least two
distinct zeros, all effectively given, we can determine effectively whether $F_1$, $F_2$ are $A$-equivalent.
\end{lem}

\begin{proof}
This is \cite[Lemma 10.4.1]{EvertseGyory16}.
\end{proof}

Henceforth, the integral domain $A$ is given effectively in the form
\[
\Zz[X_1,\ldots ,X_r]/(f_1,\ldots ,f_s)=\Zz[x_1,\ldots ,x_r]
\]
where $x_i$ is
the residue class of $X_i$ mod $(f_1,\ldots ,f_s)$ for $i=1,\ldots
,r$. Further the finite extension $G$ of the quotient field $K$
of $A$ is given in the form $K[X]/(P)$ or $K(w)$, where $w$ is the
residue class of $X$ $\mbox{mod}\, P$. The polynomial $P$ may be
represented as $b_0^{-1}\sum_{i=0}^{d}b_iX^{d-i}$
with $b_0,\ldots ,b_d$ given as polynomials in $x_1,\ldots ,x_r$
with integer coefficients.
Define
$$
u:=b_0w.
$$
Then $u$ has minimal polynomial
\begin{equation}\label{e10.5.3m}
Q=X^{d}+\sum_{i=1}^{d}b_ib_0^{d-1-i}X^{d-i}=:X^{d}+\sum_{i=1}^{d}c_iX^{d-i}
\in A[X]
\end{equation}
over $K$. Now clearly, $G=K(u)$, $u$ is integral over $A$, and
every element of $G$ can be expressed in the form
$\sum_{i=0}^{d-1}(a_i/b)u^{i}$ with $a_0,\ldots ,a_{d-1}, b\in A$,
given as polynomials with integer coefficients in
$x_1,\ldots ,x_r$.

\begin{proof}[Proof of \thmref{t10.1.1m}]
Let $A,G,n,\delta$ be effectively given and satisfy the conditions
of \thmref{t10.1.1m}.
Further, let
$\irF$ be the finite effectively determinable set
from \propref{p10.2.2m}.

Take a monic polynomial $F$ from $A[X]$ with \eqref{e10.1m}.
Then $F$ has all its zeros in $G$, say $F(X)=(X-\alpha_1)\cdots
(X-\alpha_n)$, with $\alpha_1,\ldots ,\alpha_n\in G$. By
\propref{p10.2.2m} we have
$$
\alpha_i-\alpha_j\in\irF\,\,\mbox{ for } i,j\in\left\{1,\ldots ,n\right\}\,\mbox{ with } i\neq j.
$$
Recall that $\irF$ is finite, and effectively
determinable in terms of $A$, $G$, $n$, $\delta$. For each tuple
$\left(\gamma_{ij}:\,\, i,j\in\left\{1,\ldots ,n\right\}, i\neq
j\right)$ with elements from $\irF$ we consider the
polynomials $F$ with \eqref{e10.1m} and with
$\alpha_i-\alpha_j=\gamma_{ij}$ for $i,j\in\left\{1,\ldots
,n\right\}$, $i\neq j$. That is, we consider polynomials $F$ such
that
\begin{equation}\label{e10.5.4m}
\left\{\begin{array}{l}F\in A[X], F \mbox{ monic }, \deg F=n,
D(F)=\delta,\\
F=(X-\alpha_1)\cdots (X-\alpha_n) \mbox{ for some }
\alpha_1,\ldots ,\alpha_n\in G
\\
\mbox{such that }\alpha_i-\alpha_j=\gamma_{ij} \mbox{ for }
i,j\in\left\{1,\ldots ,n\right\}, i\neq j,\end{array}\right.
\end{equation}
Our proof will be completed as follows.
We show that for each tuple $\{\gamma_{ij}\}$
it can be decided effectively whether a
polynomial $F$ with \eqref{e10.5.4m} exists.
If so, we show that the polynomials with \eqref{e10.5.4m}
lie in finitely many $A$-equivalence classes,
and determine effectively a full system of representatives for them.
Then from the union of these systems,
we extract a full system of representatives for the $A$-equivalence
classes of solutions of \eqref{e10.1m}.

Fix elements $\gamma_{ij}$ from $\irF$ ($1\leq i,j\leq n$, $i\not= j$).
Suppose there is a polynomial $F$ with \eqref{e10.5.4m}.
For this polynomial we have
\begin{equation}\label{e10.5.5m}
n\alpha_i=y+\gamma_i\,\mbox{ for } i=1,\ldots ,n,
\end{equation}
with $y=\alpha_1+\cdots +\alpha_n$,
$\gamma_i=\sum_{j=1}^{n}\gamma_{ij}$ for $i=1,\ldots ,n$. Here
$\gamma_1,\ldots ,\gamma_n$ are fixed and $y,\alpha_1,\ldots
,\alpha_n$ are variables. The number $y$ is a coefficient of $F$,
so $y\in A$. Further, if there is a polynomial $F$ with
\eqref{e10.5.4m}, then
\begin{equation}\label{e10.5.6m}
(X-\gamma_1)\cdots
(X-\gamma_n)=n^{n}F\left(\frac{X+y}{n}\right)\in A[X].
\end{equation}
The coefficients of $(X-\gamma_1)\cdots (X-\gamma_n)$ belong to
$G$. It can be checked whether they belong to $K$,
and then by means of  and by Corollary \ref{cor:E0-11} (ii), it can be checked whether they
belong to $A$. If not so, there is no polynomial with
\eqref{e10.5.4m}. So we assume henceforth that
$(X-\gamma_1)\cdots (X-\gamma_n)\in A[X]$.
Then $\gamma_1,\ldots ,\gamma_n\in A_G$.

Using Proposition \ref{cor:E0-17}, we compute a finite set of $A$-module
generators for $A_G$. From this,
we deduce a system
$\left\{\mathbf{a}_1,\ldots ,\mathbf{a}_t\right\}$ of $A$-module
generators for $A_G^n$.
The numbers $\alpha_1,\ldots ,\alpha_n$ from \eqref{e10.5.4m} are in $A_G$.
So there
are $x_1,\ldots ,x_t\in A$ such that
\begin{equation}\label{e10.5.7m}
\left(\begin{array}{c}\alpha_1\\ \vdots \\
\alpha_n\end{array}\right)=x_1\mathbf{a}_1+\cdots
+x_t\mathbf{a}_t,
\end{equation}
and we can rewrite \eqref{e10.5.5m} as
\begin{equation}\label{e10.5.8m}
x_1(n\mathbf{a}_1)+\cdots
+x_t(n\mathbf{a}_t)=y\left(\begin{array}{c}1\\
\vdots
\\ 1\end{array}\right)+\left(\begin{array}{c}\gamma_1 \\ \vdots \\
\gamma_n\end{array}\right).
\end{equation}
By linear algebra, we can determine a maximal
$K$-linearly-independent subset of $\left\{n\mathbf{a}_1,\ldots
,n\mathbf{a}_t,(1,\ldots ,1)^{T}, (\gamma_1,\ldots
,\gamma_n)^{T}\right\}$,
say $\left\{\mathbf{b}_1,\ldots ,\mathbf{b}_m\right\}$.
Further, we can compute expressions for $n\mathbf{a}_1,\ldots
,n\mathbf{a}_t,(1,\ldots ,1)^{T}, (\gamma_1,\ldots ,\gamma_n)^{T}$
as $K$-linear combinations of $\mathbf{b}_1,\ldots ,\mathbf{b}_m$.
By substituting
these into \eqref{e10.5.8m} and equating the coordinates of
\eqref{e10.5.8m}, we obtain a system of inhomogeneous linear
equations:
\begin{equation}\label{e10.5.9m}
M\mathbf{x}=\mathbf{b}\,\mbox{ in } \mathbf{x}=(x_1,\ldots
,x_t,y)^{T}\in A^{t+1}
\end{equation}
where the matrix $M$ and vector $\mathbf{b}$ have their entries in
$K$. Then using Corollary \ref{cor:E0-11} we can decide whether
\eqref{e10.5.9m} is solvable and if so, compute a solution.
Translating this back to \eqref{e10.5.8m}, we can decide whether
\eqref{e10.5.8m} is solvable and if so, compute a solution.

If \eqref{e10.5.8m} is unsolvable, then there is no polynomial $F$
with \eqref{e10.5.4m}. Assume \eqref{e10.5.8m} is solvable and
compute a solution, say $(x_{10},\ldots ,x_{t0},y_0)$ $\in
A^{t+1}$. Thus,
$\sum_{i=1}^{t}x_{i0}(n\mathbf{a}_i)-y_0(1,\ldots
,1)^{T}=(\gamma_1,\ldots ,\gamma_n)^{T}$. Put
\begin{equation}\label{e10.5.10m}
\left(\begin{array}{c}\alpha_{10} \\ \vdots \\
\alpha_{n0}\end{array}\right):=x_{10}\mathbf{a}_1+\cdots
+x_{n0}\mathbf{a}_t.
\end{equation}
Then
\begin{equation}\label{e10.5.11m}
n\alpha_{i0}=y_0+\gamma_i\,\mbox{ for } i=1,\ldots ,n\,\mbox
{ with } y_0\in A.
\end{equation}
Now let again $F$ be an arbitrary polynomial with \eqref{e10.5.4m}
and let $y$ be as in \eqref{e10.5.5m}. From \eqref{e10.5.5m},
\eqref{e10.5.11m} we infer that
\begin{equation}\label{e10.5.12m}
\alpha_i-\alpha_{i0}=\frac{y-y_0}{n}=:a\,\mbox{ for } i=1,\ldots ,n.
\end{equation}
Clearly, $a\in\frac{1}{n}A$.
Identity \eqref{e10.5.10m} implies that
$\alpha_{10},\ldots ,\alpha_{n0}\in A_G$.
Hence $a$ is
integral over $A$. So in fact, $a\in\frac{1}{n}A\cap A_K$.

By Corollary \ref{cor:E0-18}, we can compute a full system of representatives,
say $\{\theta_1,\ldots,\theta_h\}$ for
$\big(\frac{1}{n}A^+\cap\overline{A}^+\big)/A^+$.
For $j=1,\ldots, h$, put
\[
F_j(X):=(X-\alpha_{10}-\theta_j)\cdots (X-\alpha_{n0}-\theta_j).
\]
For some
$j\in\{ 1,\ldots, h\}$ we have
$a=\theta_j+c$ for some $c\in A$.
Then \eqref{e10.5.12m} implies that $\alpha_i=\alpha_{i0}+\theta_j+c$
for $i=1,\ldots, n$, and so $F(X)=F_j(X-c)$.
Hence $F$ is $A$-equivalent to $F_j$.

The polynomials $F_1,\ldots, F_h$ can be determined effectively.
Their coefficients belong to $K$ and using Corollary \ref{cor:E0-11}
we can select those polynomials that have their coefficients in $A$.
Thus, for each tuple $\{ \gamma_{ij}\}$ with $\gamma_{ij}\in\irF$
we can compute a finite system of polynomials from $A[X]$
such that every polynomial with \eqref{e10.5.4m} is
$A$-equivalent to one of them. By taking the union of these systems
for all tuples $\{ \gamma_{ij}\}$, we effectively determine a finite list of
polynomials from $A[X]$ such that every polynomial with \eqref{e10.1m} is
$A$-equivalent to at least one of them.
For each polynomial from the list we can effectively decide
whether it satisfies
\eqref{e10.1m} and if not so, remove it.
Finally, assuming the list is ordered,
by means of Lemma \ref{l10.5.1m} we can effectively decide whether a polynomial
from the list is $A$-equivalent to an earlier polynomial in the list
and if so, remove it.
This leaves us with a full system of representatives
for the $A$-equivalence classes of polynomials with
\eqref{e10.1m}.
This completes the proof of Theorem \ref{t10.1.1m}.
\end{proof}

\section{Proof of Theorem \ref{t10.1.4m}}\label{section10.6m}

Let $A$ be an integral domain finitely generated over $\Zz$,
effectively given as usual in the form
$\Zz[X_1,\ldots ,$ $X_r]/(f_1,\ldots ,f_s)=\Zz[x_1,\ldots ,x_r]$,
where $f_1,\ldots,f_s$ $\in\Zz[X_1,\ldots ,X_r]$ and where
$x_i$ is the residue class of
$X_i$ $\mbox{mod } (f_1,\ldots
,f_s)$ for $i=1,\ldots ,r$. Denote by $K$ the quotient field of
$A$. Let $\Omega$ be a \fk, effectively given in the form
$K[X]/(P)=K[\theta]$, where $P\in K[X]$ is a monic polynomial
without multiple zeros, and $\theta =\mbox{mod } P$.

We need some results from \cite[\S10.7]{EvertseGyory16}.
Using \cite[Cor. 10.7.7]{EvertseGyory16} we can construct the splitting field of
$P$ over $K$; call this $G$. By means 
of \cite[Cor. 10.7.8]{EvertseGyory16}
we can compute $w$ such that $G=K(w)$, together with the minimal
polynomial of $w$ over $K$. As was explained in
Section \ref{section10.5m}, we can compute from this another
representation for $G$ of the form $K(u)$, where $u$ is integral
over $A$, together with the monic minimal polynomial $Q$ of $u$
over $K$. Elements of $G$ are always given in the form
$\sum_{i=0}^{d-1}(a_i/b)u^{i}$ where $d=[G:K]$ and $a_0,\ldots
,a_{d-1}, b$ are elements of $A$.

The polynomial $P$ factorizes as $(X-\theta^{(1)})\cdots
(X-\theta^{(n)})$ in $G$, and by 
\cite[Cor. 10.7.8]{EvertseGyory16} we can
compute expressions of $\theta^{(1)},\ldots ,\theta^{(n)}$ as
$K$-linear combinations of $1, u,\ldots ,u^{d-1}$. With these
expressions we can compute, for any element
$\alpha=\sum_{i=0}^{n-1}c_i\theta^{i}\in\Omega$ with $c_0,\ldots
,c_{n-1}\in K$, its images
$\alpha^{(j)}=\sum_{i=0}^{n-1}c_i(\theta^{(j)})^{i}$ $(j=1,\ldots
,n)$ under the $K$-homomorphisms of $\Omega$ to $G$.

We start with a lemma.

\begin{lem}\label{l10.6.2m}
For any two effectively given $\alpha_1,\alpha_2\in O$ with
$K[\alpha_1]=K[\alpha_2]=\Omega$,
we can decide effectively whether $\alpha_1$, $\alpha_2$ are $A$-equivalent.
\end{lem}

\begin{proof}
See \cite[Lemma 10.5.1]{EvertseGyory16}.
\end{proof}

\begin{proof}[Proof of \thmref{t10.1.4m}]
Let $A,\Omega ,O$ be the effectively given integral domain, \fk, and
$A$-order in $\Omega$. Assume that
$(O\cap K)^+/A^+$ is finite. Let $\left\{\omega_1=1,\ldots
,\omega_m\right\}$ be the effectively given system of $A$-module generators
for $O$. Further, let $n=[\Omega:K]$, $n\geq 2$ and let
$\delta$ be the given element of $A$. 
Lastly, let $G$ be the field
defined above, given in the form $K(u)$ with $u$ integral over
$A$.

Recall that by Proposition \ref{th:E0-16}, we can compute an
ideal representation for the integral closure $A_K$ of $A$,
i.e., $A_K$ is effectively given as an integral domain in the usual sense.
So we can apply \propref{p10.2.2m} with $A_K$ instead of $A$.
Let $\irF'$ be the finite set $\irF$ from \propref{p10.2.2m}
but taken with $A_K$ instead of $A$.
This set
can be computed effectively in terms of
$A_K$, $G$, $\delta$, hence in terms of $A$, $\Omega$, $\delta$. 
Now if $\alpha$ is an
element of $O$ with \eqref{e10.1.6m}, i.e.,
$D_{\Omega/K}(\alpha)=\delta$, then $\alpha\in A_{\Omega}$, hence
$F_\alpha(X):=(X-\alpha^{(1)})\cdots (X-\alpha^{(n)})$ has its
coefficients in $A_K$, we have $D(F_\alpha)=\delta$, and $F_\alpha$ has its
zeros in $G$. Hence
$$
\alpha^{(i)}-\alpha^{(j)}\in\irF\,'\
\mbox{ for } i,j\in\left\{1,\ldots ,n\right\}, i\neq j.
$$

We now pick elements $\gamma_{ij}$ from $\irF\,'$
and consider the elements $\alpha$ with
\begin{equation}\label{e10.6.4m}
\left\{
\begin{array}{l}
\alpha\in O,\,\, D_{\Omega/K}(\alpha)=\delta,\\
\alpha^{(i)}-\alpha^{(j)}\in\gamma_{ij}\,\mbox{ for }
i,j\in\left\{1,\ldots ,n\right\}, i\neq j.
\end{array}
\right.
\end{equation}
We show that it can be decided effectively whether
\eqref{e10.6.4m} is solvable and if so, compute a solution of
\eqref{e10.6.4m}. Notice that \eqref{e10.6.4m} is certainly
unsolvable if $\prod_{1\leq i<j\leq n}\gamma_{ij}^{2}\neq \delta$.
Assume that $\prod_{1\leq i<j\leq n}\gamma_{ij}^{2}=\delta$. Then
the condition $D_{\Omega/K}(\alpha)=\delta$ can be dropped.
Writing $\alpha$ as $\sum_{k=1}^{m}x_k\omega_k$ with $x_1,\ldots
,x_m\in A$, we can rewrite \eqref{e10.6.4m} as
\begin{equation}\label{e10.6.5m}
\sum_{k=1}^{m}x_k\left(\omega_k^{(i)}-\omega_k^{(j)}\right)=\gamma_{ij}\,\mbox{
for } i,j\in\left\{1,\ldots ,n\right\}, i\neq j.
\end{equation}
Clearly, $(x_1,\ldots ,x_m)$ is a solution of \eqref{e10.6.5m} in
$A^{m}$ if and only if $\alpha:=\sum_{k=1}^{m}x_k\omega_k$ is a
solution of \eqref{e10.6.4m}.

By expressing $\omega_k^{(i)}-\omega_k^{(j)}$ and the numbers
$\gamma_{ij}$ as $K$-linear combinations of $1,u,\ldots ,u^{d-1}$
where $d=[G:K]$ and $u$ is the generating element of $G$ over
$K$, we can rewrite \eqref{e10.6.5m} as a system of inhomogeneous
linear equations in $A^m$ like in Corollary \ref{cor:E0-11} (ii).
Thus, it can be
decided effectively whether \eqref{e10.6.5m} is solvable, and if
so, a solution can be computed. Equivalently, it can be decided
effectively whether \eqref{e10.6.4m} is solvable and if so, a
solution can be computed.

For each choice of $\gamma_{ij}\in\irF '$ ($1\leq i,\leq n$, $i\not= j$),
we check if \eqref{e10.6.4m} is solvable and if so,
we compute a solution. Let $\mathcal{T}=\{ \alpha_1,\ldots, \alpha_g\}$
be the finite set obtained in this manner.
By Corollary \ref{cor:E0-19} (ii) we can compute a full system of representatives
$\{ \theta_1,\ldots,\theta_h\}$ for $(O\cap K)^+/A^+$.
Using Lemma \ref{l10.6.2m}, we can compute a maximal subset of
$\{ \alpha_i+\theta_j:\, i=1,\ldots, g,\, j=1,\ldots, h\}$ any two distinct
elements of which are not $A$-equivalent, say $\mathcal{U}$.
We show that $\mathcal{U}$ is a full system of representatives for the
$A$-equivalence classes of solutions of \eqref{e10.1.6m}.

Let $\alpha$ be a solution of \eqref{e10.1.6m}.
Then $\alpha$ satisfies \eqref{e10.6.4m} for certain $\gamma_{ij}\in\irF '$.
Let $\alpha_0$ be an element from $\mathcal{T}$ satisfying
\eqref{e10.6.4m} for these $\gamma_{ij}$. Then
$\alpha^{(i)}-\alpha^{(j)}=\alpha_0^{(i)}-\alpha_0^{(j)}$ for
$i,j\in\left\{1,\ldots ,n\right\}$, hence
$$\alpha^{(1)}-\alpha_0^{(1)}=\cdots =\alpha^{(n)}-\alpha_0^{(n)}.$$
It follows that $\alpha-\alpha_0=:a\in O\cap K$. Hence
$a=\theta_j+c$ for some $j\in\{ 1,\ldots, h\}$ and $c\in A$, and so,
$\alpha =\alpha_0+\theta_j+c$. Now clearly,
$\alpha$ is strongly $A$-equivalent to an element of $\mathcal{U}$.
This completes our proof
of \thmref{t10.1.4m}.
\end{proof}

\end{document}